\documentclass[a4paper,11pt,UTF8]{article}
\usepackage{amsfonts}
\usepackage{mathrsfs}
\usepackage{amsmath,amssymb,amscd,bm,bbm,amsthm,mathrsfs,verbatim}
\usepackage{indentfirst}
\usepackage[all,pdf]{xy}
\usepackage{geometry}
\usepackage{tabularx}
\usepackage{tikz}
\usepackage{tikz-cd}
\usetikzlibrary{arrows.meta}
\usetikzlibrary{babel}
\usepackage{diagbox}
\usepackage{graphicx}
\usepackage{stmaryrd}
\usepackage{enumerate}
\usepackage{mathtools}
\usepackage{multicol}
\usepackage{multirow}
\usepackage{float}
\usepackage[utf8]{inputenc}
\providecommand{\keywords}[1]{\textbf{\textit{Key words:}} #1}
\numberwithin{equation}{section}
\usepackage{cite}
\usepackage[comma,numbers,sort&compress]{natbib}

\usepackage{hyperref}
\usepackage[T1]{fontenc}
\usepackage{nameref,cleveref}
\usepackage{nicefrac,xfrac}

\theoremstyle{plain}
\newtheorem{thm}{Theorem}[section]
\newtheorem{lem}{Lemma}[section]
\newtheorem{prop}{Propertion}[section]
\newtheorem{cor}{Corollary}[section]

\theoremstyle{definition}
\newtheorem{defn}{Definition}[section]

\newtheorem{prob}{Problem}[section]

\theoremstyle{Remark}

\begin{document}
\title{\textbf{On uniform and nonhomogeneous vector bundles over Grassmannians}}
\author{Rong Du, \thanks{School of Mathematical Sciences
		Shanghai Key Laboratory of PMMP,
		East China Normal University,
		Rm. 312, Math. Bldg, No. 500, Dongchuan Road,
		Shanghai, 200241, P. R. China,
		rdu@math.ecnu.edu.cn.
	}
	Yiting Wang
	\thanks{School of Mathematical Sciences
		Shanghai Key Laboratory of PMMP,
		East China Normal University,
		No. 500, Dongchuan Road,
		Shanghai, 200241, P. R. China,
		1127651752@qq.com
	}
	and Dazhi Zhang
	\thanks{School of Mathematical Sciences
		Shanghai Key Laboratory of PMMP,
		East China Normal University,
		No. 500, Dongchuan Road,
		Shanghai, 200241, P. R. China,
		956502686@qq.com
		All the authors are sponsored by Innovation Action Plan (Basic research projects) of Science and Technology Commission of Shanghai Municipality (Grant No. 21JC1401900), Natural Science Foundation of Chongqing, China (general program, Grant No. CSTB2023NSCQ-MSX0334) and Science and Technology Commission of Shanghai Municipality (Grant No. 18dz2271000).
	}
}

\date{}
\maketitle
	
\begin{abstract}
We demonstrate the existence of a uniform and nonhomogeneous vector bundle $E$ of rank $(n-d)(m+1)-1$ over Grassmannian $\mathbb{G}(d,n)$, where $m>d$ and $1\le d \le n-d-1$ with a $\mathbb{P}$-homogeneity degree $h(E)=d$. Particularly, we establish an upper bound of $3(n-d)-2$ for the uniform-homogeneous shreshold of $\mathbb{G}(d,n)$. Additionally, we construct indecomposable uniform vector bundles of rank $(d+2)(n-d)+d-2+\sum\limits_{i=0}^p\tbinom{d-1+p-i}{p-i}(1+i)-\tbinom{p+d}{p}$ that are nonhomogeneous over $\mathbb{G}(d,n)$.
\end{abstract}

\keywords{uniform vector bundle, nonhomogeneous vector bundle,  Grassmannian}

\section{Introduction}
Holomorphic vector bundles over a complex projective manifold constitute fundamental research subjects in algebraic geometry and complex geometry. According to Grothendieck's theorem, any holomorphic vector bundle over a projective line decomposes as a direct sum of line bundles. However, in the case where the dimension of the projective space exceeds one, the situation becomes intricate. Consequently, the decomposition of holomorphic vector bundles on higher-dimensional projective spaces has long been a central concern within the realm of problems concerning vector bundles in algebraic geometry and complex geometry. For brevity, we refer to vector bundles of rank $r$ as $r$-bundles in the context.

One of the extensively studied classes pertains to uniform vector bundles over projective spaces; these are bundles whose splitting types remain consistent regardless of the
chosen lines. Much effort has been dedicated to classifying uniform vector bundles. For detailed insights, refer to the introductions in \cite{M-O-C}, \cite{du2023vector}, or \cite{D-F-G}. Another notable class is that of homogeneous vector bundles, which retain their structure under the automorphism pullback of the projective spaces. It is evident that homogeneous vector bundles are also uniform. Initially, there was a prevailing belief, supported by the classification results of uniform $2$-bundles and $3$-bundles over projective spaces (\cite{van1971uniform}, \cite{sato1976uniform}, and \cite{elencwajg1978fibres}), that all uniform bundles are homogeneous. However, in 1979, Elencwajg refuted this belief by constructing a uniform $4$-bundle over $\mathbb{P}^2$ that was not homogeneous (\cite{elencwajg1979fibres}). Subsequently, Hirschowitz produced examples of uniform nonhomogeneous bundles of rank $3n-1$ over $\mathbb{P}^n$ for $n\ge 3$ (\cite{2011Vector}). In 1980, Dr{\'e}zet demonstrated the existence of uniform nonhomogeneous bundles of rank  $2n$ over $\mathbb{P}^n$ for $n\ge 3$ (\cite{drezet1980exemples}).

Determining the integers $t$ such that uniform $t$-bundles over $\mathbb{P}^n$ are homogeneous is a fundamental question. Specifically, we aim to find the largest integer $UH(\mathbb{P}^n)$ for which uniform $k$-bundles over $\mathbb{P}^n$ are homogeneous, for every $k\leq UH(\mathbb{P}^n)$. This integer is referred to as the uniform-homogeneous threshold of $\mathbb{P}^n$. However, it remains a longstanding conjecture that every uniform vector bundle of rank  $r<2n$ is homogeneous, suggesting $UH(\mathbb{P}^n)=2n-1$. According to the classification theorem of uniform $(n+1)$-bundles over $\mathbb{P}^n$ by Ellia (\cite{Ell}) and Ballicao (\cite{Bal}), we have the inequality $n+1\leq UH(\mathbb{P}^n)\leq 2n-1$.

Extending this analysis to general rational homogeneous spaces $X$ with Picard number $1$, we define the uniform-homogeneous threshold $UH(X)$ of $X$. Among such spaces, Grassmannians stand out as the simplest examples, excluding projective spaces. In 1985, Guyot (\cite{guyot1985caracterisation}) demonstrated that when $r\leq d+1$, uniform $r$-bundles over $\mathbb{G}(d,n)$ ($d+1\leq n-d$) are homogeneous. Recently, Zhou and the first author provided further insight by establishing that uniform $(d+2)$-bundles over $\mathbb{G}(d,n)$ are also homogeneous, as per their classification theorem (\cite{D-Z}).

In this research paper, the authors have presented a construction of indecomposable uniform vector bundles over the Grassmannian $\mathbb{G}(d, n)$ of a specific rank. The rank is given by the expression $$(d+2)(n-d)+d-2+\sum\limits_{i=0}^p\tbinom{d-1+p-i}{p-i}(1+i)-\tbinom{p+d}{p}$$ where the splitting type is specified as $(1,\dots,1,0,\dots,0)$. Interestingly, these constructed bundles are shown to be nonhomogeneous over the Grassmannian.

Furthermore, the authors establish a theorem (referred to as Theorem \ref{MThm} in the paper) concerning the $\mathbb{P}$-homogeneity  (see Definition \ref{Ph}) of vector bundles.

\begin{thm}\emph{(Theorem \ref{MThm})}
	For every integer $d$ $(1\le d \le n-d-1)$, there exists a vector bundle $E$ of rank $(n-d)(m+1)-1~(m>d)$ over $\mathbb{G}(d,n)$ with the degree of $\mathbb{P}$-homogeneity $h(E)=d$. In particular,
	$E$ is uniform of splitting type $(1,\dots,1,0,\dots,0)$ but not nonhomogeneous.
\end{thm}

\begin{cor}
	The uniform-homogeneous shreshold of $\mathbb{G}(d,n)$ $$UH(\mathbb{G}(d,n))\le 3(n-d)-2.$$
	\end{cor}
Thus, we know that $d+2\le UH(\mathbb{G}(d,n))\le 3(n-d)-2$ by \cite{D-Z}.

\section{Preliminaries}
Let $V$ be a $(n+1)$-dimensional vector space over $K$, where $K$ is an algebraically closed field of characteristic $0$ (or simply $\mathbb{C}$). We denote by $\mathbb{G}$ the Grassmannian $\mathbb{G}(d,n)$ of $(d+1)$-dimensional linear subspaces in $V$, where a point $x$ of $\mathbb{G}(d,n)$ corresponds to a $(d+1)$-dimensional linear subspace $X$ of $V$. The projective space $\mathbb{P}^n = \mathbb{P}(V)$ is the extreme Grassmannian $\mathbb{G}(0,n)$. Additionally, the Grassmannian $\mathbb{G}(d,n)$ is isomorphic to $\mathbb{G}(n-d-1,n)$, which has dimension $(d+1)(n-d)$.

In this paper, we focus on the Grassmannian $\mathbb{G}(d,n)$ with $1\leq d\leq n-d-1$.
Denote by $\mathcal{O}_{\mathbb{G}} \otimes V$ the trivial bundle on $\mathbb{G}$, where the fiber at every point is the vector space $V$. We use $\mathcal{R}_\mathbb{G}$ to represent the $(d+1)$-rank universal subbundle on $\mathbb{G}$, with the fiber at a point $x\in \mathbb{G}$ being the subspace $X$ itself. Similarly, $\mathcal{Q}_\mathbb{G}$ denotes the $(n-d)$-rank universal quotient bundle on $\mathbb{G}$, with the fiber at a point $x\in \mathbb{G}$ being the quotient space $V/X$. This leads to the following exact sequence:

%
\begin{equation}\label{universal sequence}
\xymatrix{
	0\ar[r] & \mathcal R_\mathbb{G}\ar[r]
	&\mathcal O_\mathbb{G}^{\oplus(n+1)}\ar[r]
	&\mathcal Q_\mathbb{G}\ar[r] &0},
\end{equation}
which is called the universal sequence on $\mathbb{G}(d,n)$. It is also the Euler sequence when $d=0$, i.e.,
$\mathcal R_{\mathbb{P}(V)}=\mathcal O_{\mathbb{P}(V)}(-1)$ and
$\mathcal Q_{\mathbb{P}(V)}=\mathcal T_{\mathbb{P}(V)}(-1)$, where $\mathcal T_{\mathbb{P}(V)}$ is the tangent bundle on $\mathbb{P}(V)$.

Consider a set of $t$ distinct integers $\underline{d}\coloneqq 1\le d_1< d_2< \dots < d_t\le n$. The partial flag variety $F(d_1, \dots, d_t)$ consists of partial flags of type $\underline{d}$, which are sequences $V_{d_1}\subset V_{d_2}\subset \dots \subset V_{d_t}$, where $\text{dim}(V_{d_i}) = d_i$ for $1 \leq i \leq t$. When $t=1$, this corresponds to the Grassmannian variety defined earlier.
Let's examine the flag variety $F(d_1, d_1+1)$. The subvariety ${V_{d_1+1} \mid V_{d_1} \subset V_{d_1+1} \subset V}$ within $\mathbb{G}(d,n)$ is isomorphic to $\mathbb{P}^{n-d_1}$.
Additionally, we utilize the flag variety $F(d_1, d_2)$, which parameterizes pairs $(V_{d_1}, V_{d_2})$ of linear subspaces in $V$ such that $V_{d_1} \subset V_{d_2} \subset V$. Its dimension is given by $d_1(d_2-d_1) + d_2(n+1-d_2)$.
Considering a flag $V_d \subset V_{d+2} \subset V$, the subvariety ${V_{d+1} \mid V_d \subset V_{d+1} \subset V_{d+2}}$ within $\mathbb{G}(d,n)$ is isomorphic to $\mathbb{P}^1$ \cite{guyot1985caracterisation}. Hence, the flag variety $F(d,d+2)$ represents the variety of lines in $\mathbb{G}(d,n)$.
More generally, for a flag $F(d,d+k+1)$, the subvariety ${V_{d+1} \mid V_d \subset V_{d+1} \subset V_{d+k+1}}$ within $\mathbb{G}(d,n)$ is isomorphic to $\mathbb{P}^{k}$, where $1 \leq k \leq n-d-1$. Thus, the flag variety $F(d,d+k+1)$, with $1 \leq k \leq n-d-1$, represents the variety of $\mathbb{P}^{k}$-planes in $\mathbb{G}(d,n)$. Consequently, we obtain several projective spaces $\mathbb{P}^{k}$, where $1 \leq k \leq n-d$, within the Grassmannian $\mathbb{G}(d,n)$.


We revisit some definitions concerning an $r$-bundle $E$ defined over $\mathbb{G}(d,n)$. According to Grothendieck's theorem, there exists an $r$-tuple
$$a_E(l)=(a_1(l),\dots,a_r(l));\quad a1(l)\ge\dots\ge ar(l)$$
such that $E|l \cong \bigoplus\limits_{i=1}^r \mathcal O_{l} (a_i(l)) $.
This $r$-tuple $a_E(l)$ is termed the splitting type of $E$ on $l$.
The bundle $E$ is termed uniform if $a_E(l)$ remains independent of the choice of $l$ in $\mathbb{G}(d,n)$. For instance, $\mathcal R_\mathbb{G}$ and $\mathcal Q_\mathbb{G}$ are uniform and exhibit splitting types $(0,\dots,0,-1)$ and $(1,0,\dots,0)$ respectively.
An $r$-bundle $E$ over $\mathbb{G}(d,n)$ earns the label "homogeneous" if for every automorphism $t \in Aut(\mathbb{G}(d,n))$, we have $t^*E \cong E$. It's evident that if $E$ is homogeneous, so is its dual $E^*$. It's worth noting that while homogeneous bundles are necessarily uniform, the converse doesn't always hold true.

%

To elucidate the distinction between homogeneous bundles and uniform bundles, we introduce a novel concept: $\mathbb{P}^k$-homogeneous vector bundles. This definition serves as a broadened interpretation of $k$-homogeneity of vector bundles over projective spaces, as expounded in Section 3.3 of \cite{2011Vector}.

\begin{defn}
An $r$-bundle $E$ is called $\mathbb{P}^k$-homogeneous if for any two linear embeddings $$\phi_1,\phi_2:\mathbb{P}^k\hookrightarrow \mathbb{G}(d,n),$$
we have $\phi_1^*E\cong \phi_2^*E$.
\end{defn}
A bundle is $\mathbb{P}^1$-homogeneous precisely when it is uniform. We see that $\mathbb{P}^{k+1}$-homogeneous bundles are $\mathbb{P}^k$-homogeneous.

\begin{defn}\label{Ph}
We call
\[h(E)={}\text{max}\{k\ |\ 0\leq k\leq  n-d; E\text{ is }\mathbb{P}^k\text{- homogeneous}  \}\]
the degree of $\mathbb{P}$-homogeneity of $E$.
\end{defn}

A uniform $r$-bundle on $\mathbb{G}(d,n)$ with $r\leq d$ decomposes as a direct sum of line bundles, rendering it homogeneous as noted in \cite{du2023vector}. Consequently, the Hilbert polynomial $h(E)$ for an $r$-bundle $E$ with $r\leq d$ on $\mathbb{G}(d,n)$ can only take on values of 0 or $d+1$.
An $r$-bundle $E$ is termed indecomposable if it cannot be expressed as the direct sum $E=F\oplus G$ of two proper subbundles $F,G \subset E$. On the other hand, it is called simple if $h^0(\mathbb{G},E^*\otimes E)=1$. Notably, simple bundles are inherently indecomposable.

%

In addition, we can identify the Grassmiann $\mathbb{G}(d,n)$ with the homogeneous space $GL_{n+1}(K)/P$ where
\begin{equation*}
	P=\left\{
	\left(
	\begin{array}
		{*{20}{c}}
		{h_{1}}&{h_{2}} \\
		{0_{(n-d)\times (d+1)}}&{h_{4}} \\
	\end{array}\right)
	\in GL_{n+1}(K) \right\}
\end{equation*}
is the maximal parabolic subgroup of $GL_{n+1}(K)$.

Consider the action of $P$ on the right of the vector space $V$. Define
\[ GL_{n+1}(K)\times_P V \coloneqq(GL_{n+1}(K)\times V)/ \sim, \]
where "$\sim$" is the equivalence relation $(g,v)\sim (gp^{-1},pv)$ for every
$g\in GL_{n+1}(K)$, $p\in P$, and $v\in V$. Then, the mapping
\[ \begin{array}{*{20}{c}}
    {\pi_V:} & {GL_{n+1}(K)\times_P V} & {\longrightarrow} & {GL_{n+1}(K)/P\cong \mathbb{G}(d,n)} \
    {} & {(g,v)}& {\longmapsto}& {gP} \
\end{array} \]
is referred to as a vector bundle associated with $V$. The notation $[g,v]\in GL_{n+1}(K)\times_P V$ represents the equivalence class of $(g,v)\in GL_{n+1}(K)\times V$ under the relation $\sim$.


A vector bundle $\pi : E \rightarrow \mathbb{G}(d,n)$  is called $GL_{n+1}(K)$-homogeneous if
$E$ has a $GL_{n+1}(K)$-action and $\pi$ is $GL_{n+1}(K)$-equivariant, i.e.,
for every $e\in E$, $\pi (g\cdot e)=g\cdot \pi (e)$.
Then the following diagram
$$\begin{CD}
	GL_{n+1}(K)\times E @> >> E\\
	@V  V V @VV  V\\
	GL_{n+1}(K)\times GL_{n+1}(K)/P @>> > GL_{n+1}(K)/P
\end{CD} $$
commutes.
Indeed, as outlined in \cite{ottaviani1995rational}, the definitions of $GL_{n+1}(K)$-homogeneous bundles and homogeneous bundles are synonymous. Consequently, the tangent bundle $\mathcal{T}_{\mathbb{G}(d,n)} \cong \mathcal R_\mathbb{G}^* \otimes \mathcal Q_\mathbb{G}$ is both $GL_{n+1}(K)$-homogeneous and homogeneous.
Furthermore, the homogeneous bundle $E$ is called irreducible if $V$ constitutes an irreducible $P$-module. This notion underscores the structure and properties of bundles within the framework of the associated group action.

Consider any $r$-bundle $E$ defined over $\mathbb{G}(d,n)$. A weight $\underline{v}=(v_1, v_2, \dots, v_r)\in \mathbb{Z}^r$ is called dominant if it satisfies the condition:
\[ v_1 \geq v_2 \geq \dots \geq v_r. \]
We denote the set of all dominant weights in $\mathbb{Z}^r$ as $\mathbb{Z}_{\text{dom}}^r$. Let $\mathrm{S}_{\underline{v}} E$ represent the outcome of applying the Schur functor $\mathrm{S}_{\underline{v}}$ associated with the weight $\underline{v}$ to the bundle $E$. Specifically, if $\underline{v}=(p)$, then $\mathrm{S}_{\underline{v}} E \cong S^p E $, where $S^pE$ denotes the $p^{th}$ symmetric product of $E$.
Each irreducible homogeneous bundle over $\mathbb{G}(d,n)$ takes the form $\mathrm{S}_{\underline{u}} \mathcal Q_\mathbb{G}^* \otimes
\mathrm{S}_{\underline{v}} \mathcal R_\mathbb{G}^* $, where $\underline{u} \in \mathbb{Z}_{\text{dom}}^{n-d}$ and $\underline{v} \in \mathbb{Z}_{\text{dom}}^{d+1}$, as asserted in \cite{Brown2018FlagVA}. Notably, both $S^p \mathcal R_\mathbb{G}$ and $S^p \mathcal Q_\mathbb{G}$ exhibit homogeneity.
In characteristic $0$ fields, such as $K$, the Borel-Bott-Weil theorem stands as a potent tool for computing the cohomology of homogeneous bundles, as referenced in \cite{ottaviani1995rational, Brown2018FlagVA}. This theorem empowers comprehensive analyses of the properties and structures of homogeneous bundles.
Then we have
$$H^i(\mathbb{G},  \mathcal O_{\mathbb{G}})= \left\{
\begin{array}{ll}
	K &  ,i=0 \\
	0 &  ,otherwise \\
\end{array} ,
\right.   $$
$$H^i(\mathbb{G}, S^p \mathcal Q_\mathbb{G})= \left\{
\begin{array}{ll}
	S^pV &  ,i=0 \\
	0 &  ,otherwise \\
\end{array} ,
\right.   $$
$$H^i(\mathbb{G}, S^p \mathcal R^*_\mathbb{G})= \left\{
\begin{array}{ll}
	S^pV^* &  ,i=0 \\
	0 &  ,otherwise \\
\end{array} ,
\right.   $$
\[H^i(\mathbb{G}, S^p \mathcal Q^*_\mathbb{G})= \left\{
\begin{array}{ll}
	S_{(p-d-1,\underbrace{1,\dots,1}_{n-d})}V^*  &  ,i=d+1,p\ge d+2 \\
	0 &  ,otherwise \\
\end{array} ,
\right.   \]
\[H^i(\mathbb{G}, S^p \mathcal R_\mathbb{G})= \left\{
\begin{array}{ll}
	S_ {(p+d-n,\underbrace{1,\dots,1}_{n-d})}V  &  ,i=n-d,p\ge n-d+1 \\
	0 &  ,otherwise \\
\end{array} .
\right.   \]

\section{A $\mathbb{P}^d$-homogeneous bundle over $\mathbb{G}(d,n)$}
Let's examine the universal sequence \eqref {universal sequence} defined over $\mathbb{G}(d,n)$. At a point $x$ within $\mathbb{G}(d,n)$, the fiber of the universal quotient bundle $\mathcal Q_\mathbb{G}$, denoted by $\mathcal Q_{\mathbb{G}}(x)$, is expressed as:
\[ \mathcal Q_{\mathbb{G}}(x) = V/K<e_0,e_1,\cdots,e_d>, \]
where $e_0,e_1,\cdots,e_d$ constitutes a basis of the corresponding $(d+1)$-dimensional subspace $X$ in $V$.
Let's select $m+1~(m>d)$ linearly independent vectors $ w_0,w_1,\cdots,w_m \in V $. These vectors establish certain sections $s_{wi}\in H^0(\mathbb{G},\mathcal{Q}_{\mathbb{G}})$, which, at the point $x$, are articulated as follows:
\begin{equation}\label{swidef}
    s_{w_i}(x) = w_i/K<e_0,e_1,\cdots,e_d>  \ \in \mathcal{Q}_{\mathbb{G}}(x).
\end{equation}
It's crucial to note that these sections $s_{w_i}$ do not share any common zeros. Suppose there exists a point $x_0\in \mathbb{G}(d,n)$ such that $w_i\in  K<e_0,e_1,\cdots,e_d>$ for all $i\in \{0,1,\cdots,m\}$. This assumption, however, contradicts the linear independence of $w_0,w_1,\cdots,w_m$. Consequently, these sections define a rank-$1$ trivial subbundle in $\mathcal{Q}_\mathbb{G}^{\oplus(m+1)}$, expressed as:
\[ \mathcal{O}\mathbb{G} \xhookrightarrow{(s_{w_0},\cdots,s_{w_m})} \mathcal{Q}_\mathbb{G}^{\oplus(m+1)}. \]
Let $E$ be the quotient bundle, and as a result, we obtain the following exact sequence:
\begin{align}\label{seq1}
    \xymatrix{
        0 \ar[r] &\mathcal{O}_\mathbb{G} \ar[r]
        & \mathcal{Q}_\mathbb{G}^{\oplus(m+1)} \ar[r]
        &E \ar[r] &0
    }.
\end{align}
It's evident from this sequence that $E$ is of rank $(n-d)(m+1)-1$.

Next we will determine the degree of $\mathbb{P}$-homogeneity of $E$.
\begin{lem}\label{lemw1}
	Let $W_0=Kw_0+Kw_1+\cdots+Kw_m\subseteq K^{n-d+1}$ be the subspace spanned by the vectors $w_0,w_1,\cdots,w_m$ and $W\subseteq K^{n-d+1}$ be a $(k+1)$-dimensional subspace of $K^{n-d+1}$.
	\begin{enumerate}
		\item[i)] If $W_0\nsubseteq W$, then
		\[E|_{\mathbb{P}(W)}\cong\mathcal{T}_{\mathbb{P}(W)}(-1)^{\oplus(m+1)}
		\oplus\mathcal{O}_{\mathbb{P}(W)}^{(n-d-k)(m+1)-1};\]
		\item[ii)]If $W_0\subseteq W$, then
		\[E|_{\mathbb{P}(W)}\cong E'\oplus \mathcal{O}_{\mathbb{P}(W)}^{(n-d-k)(m+1)}\]
		with a bundle $E'$ over $\mathbb{P}(W)$ such that $h^0(\mathbb{P}(W),E'^{\ast})=0$.
	\end{enumerate}
\end{lem}
\begin{proof}
\textit{i)} If $W_0$ is not a subset of $W$, then at least one vector, denoted as $w_0$, lies outside $W$. Consequently, the section
\[s_{w_0}|{\mathbb{P}(W)} \in H^0(\mathbb{P}(W),\mathcal{Q}_\mathbb{G}|_{\mathbb{P}(W)})\]
is non-vanishing everywhere. Otherwise, if there existed a point $x\in \mathbb{P}(W)$ such that $s_{w_0}|_{\mathbb{P}(W)}(x)=0$, it would imply $w_0$ belongs to $W$, contradicting the assumption $w_0 \notin W$. Consequently, we establish an exact sequence:
	$$ 0\longrightarrow \mathcal{O}_{\mathbb{P}(W)}\xrightarrow{s_{w_0}|_{\mathbb{P}(W)}} \mathcal{Q}_\mathbb{G}|_{\mathbb{P}(W)}
	\longrightarrow L \longrightarrow 0.$$
	
	Together with the exact sequence
	\begin{align}\label{seq2}
		0\longrightarrow \mathcal{Q}_{\mathbb{P}(W)}\longrightarrow \mathcal{Q}_\mathbb{G}|_{\mathbb{P}(W)}\longrightarrow
		\mathcal{O}_{\mathbb{P}(W)}^{\oplus(n-d-k)} \longrightarrow 0,
	\end{align}
	there is a commutative diagram in which rows and columns are exact:
	\[
	\xymatrix{
		&         & 0 \ar[d]   & 0 \ar[d]  &   \\
		&         & \mathcal{Q}_{\mathbb{P}(W)} \ar[d] \ar@{=}[r]  &
		\mathcal{T}_{\mathbb{P}(W)}(-1) \ar[d]  &   \\
		0  \ar[r] & \mathcal{O}_{\mathbb{P}(W)} \ar@{=}[d]\ar[r]^{s_{w_0}|_{\mathbb{P}(W)}}
		&\mathcal{Q}_{\mathbb{G}}|_{\mathbb{P}(W)} \ar[d] \ar[r] & L \ar[d] \ar[r] & 0  \\
		0 \ar[r]  &  \mathcal{O}_{\mathbb{P}(W)}\ar[r]&  \mathcal{O}_{\mathbb{P}(W)}^{\oplus(n-d-k)}\ar[d]\ar[r]
		& L'\ar[d] \ar[r] & 0  \\
		&         & 0   & 0   &                                         }
	\]
	where $L'=\mathcal{O}_{\mathbb{P}(W)}^{\oplus(n-d-k-1)} $.
	It follows that
	\begin{equation*}
		L=\mathcal{T}_{\mathbb{P}(W)}(-1)\oplus\mathcal{O}_{\mathbb{P}(W)}^{\oplus(n-d-k-1)}
	\end{equation*}
	as $H^1(\mathbb{P}(W),\mathcal{T}_{\mathbb{P}(W)}(-1))=0$ .
	
	We now consider the commutative diagram
	\[	\xymatrix{
		&    &     & 0 \ar[d]   & 0 \ar[d]  &   \\
		&     &    & \mathcal{Q}_{\mathbb{G}}^{\oplus m}|_{\mathbb{P}(W)} \ar[d] \ar@{=}[r]  &
		\mathcal{Q}_{\mathbb{G}}^{\oplus m}|_{\mathbb{P}(W)}\ar[d]  &   \\
		0  \ar[r] & \mathcal{O}_{\mathbb{P}(W)} \ar@{=}[d]\ar[rr]^{s_{w_i}|_{\mathbb{P}(W)}}
		&&\mathcal{Q}_{\mathbb{G}}^{\oplus(m+1)}|_{\mathbb{P}(W)} \ar[d] \ar[r] &E|_{\mathbb{P}(W)} \ar[d] \ar[r] & 0  \\
		0 \ar[r]  &  \mathcal{O}_{\mathbb{P}(W)}\ar[rr]^{s_{w_0}|_{\mathbb{P}(W)}} &  & \mathcal{Q}_{\mathbb{G}}|_{\mathbb{P}(W)}\ar[d]\ar[r]& L\ar[d] \ar[r] & 0  \\
		&        & & 0   & 0 .  &                                        }
	\]
	Let's assert that the sequence on the right-hand column splits, meaning $E|_{\mathbb{P}(W)}$ is isomorphic to $L\oplus \mathcal{Q}_{\mathbb{G}}^{\oplus m}|_{\mathbb{P}(W)}$. To prove this assertion, we only need to verify that $H^1(\mathbb{P}(W), L^\ast\otimes\mathcal{Q}_{\mathbb{G}}|_{\mathbb{P}(W)})$ equals $0$. By applying the functor $-\otimes L^*$ to the exact sequence \eqref{seq2}, we derive the following exact sequence:
	\[0\longrightarrow\mathcal{T}_{\mathbb{P}(W)}(-1)\otimes L^\ast\longrightarrow \mathcal{Q}_\mathbb{G}|_{\mathbb{P}(W)}\otimes L^\ast\longrightarrow
	\mathcal{O}_{\mathbb{P}(W)}^{\oplus(n-d-k)}\otimes L^\ast\longrightarrow 0, \]
	which induces a long exact sequence
	\begin{align*}
		0\longrightarrow& H^0(\mathbb{P}(W),\mathcal{T}_{\mathbb{P}(W)}(-1)\otimes L^\ast)\longrightarrow
		H^0(\mathbb{P}(W),\mathcal{Q}_\mathbb{G}|_{\mathbb{P}(W)}\otimes L^\ast)\longrightarrow \\
		& H^0(\mathbb{P}(W),\mathcal{O}_{\mathbb{P}(W)}^{\oplus(n-d-k)}\otimes L^\ast)
		\xrightarrow{\Delta^0}   H^1(\mathbb{P}(W),\mathcal{T}_{\mathbb{P}(W)}(-1)\otimes L^\ast)\longrightarrow \\
		&H^1(\mathbb{P}(W),\mathcal{Q}_\mathbb{G}|_{\mathbb{P}(W)}\otimes L^\ast)\longrightarrow
		H^1(\mathbb{P}(W),\mathcal{O}_{\mathbb{P}(W)}^{\oplus(n-d-k)}\otimes L^\ast)\longrightarrow \cdots.
	\end{align*}
	
	If $H^1(\mathbb{P}(W),\mathcal{T}_{\mathbb{P}(W)}(-1)\otimes L^\ast)$ and $H^1(\mathbb{P}(W),\mathcal{O}_{\mathbb{P}(W)}^{\oplus(n-d-k)}\otimes L^\ast)$ are both equal to 0, the assert holds immediately. It is true, since we can compute them as follows
	\begin{align*}
		& H^1(\mathbb{P}(W),\mathcal{T}_{\mathbb{P}(W)}(-1)\otimes L^\ast)= H^1(\mathbb{P}(W), \mathcal{O}_{\mathbb{P}(W)}\oplus \mathcal{T}_{\mathbb{P}(W)}(-1)^{\oplus(n-d-k-1)} )=0 ,\\
		& H^1(\mathbb{P}(W),\mathcal{O}_{\mathbb{P}(W)}^{\oplus(n-d-k)}\otimes L^\ast)=H^1(\mathbb{P}(W), \Omega_{\mathbb{P}(W)}^1(1)^{\oplus(n-d-k)}\oplus\mathcal{O}_{\mathbb{P}(W)}^{\oplus(2n-2d-2k-1)})=0.
	\end{align*}
	
	Note that the exact sequence \eqref{seq2} is also split. Altogether we then have
	\begin{align*}
		E|_{\mathbb{P}(W)} & \cong (\mathcal{T}_{\mathbb{P}(W)}(-1)\oplus\mathcal{O}_{\mathbb{P}(W)}^{\oplus(n-d-k-1)})
		\oplus(\mathcal{T}_{\mathbb{P}(W)}(-1)\oplus \mathcal{O}_{\mathbb{P}(W)}^{\oplus(n-d-k)}) ^{\oplus m} \\
		& = \mathcal{T}_{\mathbb{P}(W)}(-1)^{\oplus(m+1)}\oplus \mathcal{O}_{\mathbb{P}(W)}^{\oplus(n-d-k-1)(m+1)-1}.
	\end{align*}
	
	\textit{ii)} If $W_0\subseteq W$, the sections $s_{w_i}(i=0,1,\dots,m)$ can be regarded as sections in $\mathcal{Q}_{\mathbb{P}(W)}$.
	Then we get the following commutative diagram:
	\[
	\xymatrix{
		&         && 0 \ar[d]   & 0 \ar[d]  &   \\
		0 \ar[r]  &\mathcal{O}_{\mathbb{P}(W)}\ar@{=}[d]\ar[rr]^{(s_{w_0},\cdots,s_{w_m})}    &      & \mathcal{Q}_{\mathbb{P}(W)}^{\oplus (m+1)} \ar[d] \ar[r]  &
		E'\ar[d]\ar[r]  & 0  \\
		0  \ar[r] & \mathcal{O}_{\mathbb{P}(W)}\ar[rr]^{(s_{w_0},\cdots,s_{w_m})}
		& &\mathcal{Q}_{\mathbb{G}}^{\oplus(m+1)}|_{\mathbb{P}(W)} \ar[d] \ar[r] &E|_{\mathbb{P}(W)} \ar[d] \ar[r] & 0  \\
		&    &&  \mathcal{O}_{\mathbb{P}(W)}^{\oplus(n-d-k)(m+1)}\ar[d]\ar@{=}[r]&
		\mathcal{O}_{\mathbb{P}(W)}^{\oplus(n-d-k)(m+1)}\ar[d]  &  \\
		&        & & 0   & 0 .  &                                       }
	\]
	
	Taking the dual of the top exact sequence, we obtain an exact sequence
	\[0\longrightarrow E'^\ast\longrightarrow \Omega_{\mathbb{P}(W)}^1(1)^{\oplus(m+1)}
	\longrightarrow\mathcal{O}_{\mathbb{P}(W)}\longrightarrow 0,\]
which determines a long exact sequence
	\[ 0\longrightarrow H^0(\mathbb{P}(W),E'^\ast)\longrightarrow H^0(\mathbb{P}(W),\Omega_{\mathbb{P}(W)}^1(1)^{\oplus (m+1)})
	\longrightarrow \cdots.\]
	It follows that
	$H^0(\mathbb{P}(W),E'^\ast)=0$, since $H^1(\mathbb{P}(W), \Omega_{\mathbb{P}(W)}^1(1))=0$.

	Taking the cohomology of the top exact sequence directly, we have the following exact sequence
	\[\cdots \longrightarrow H^1(\mathbb{P}(W),\Omega_{\mathbb{P}(W)}^1(1)^{\oplus(m+1)})
	\longrightarrow H^1(\mathbb{P}(W),E')\longrightarrow H^2(\mathbb{P}(W),\mathcal{O}_{\mathbb{P}(W)})\longrightarrow \cdots.\]
	Then $h^1(\mathbb{P}(W),E')=0$ as the result of
	$h^1(\mathbb{P}(W),\Omega_{\mathbb{P}(W)}^1(1))=h^2(\mathbb{P}(W),\mathcal{O}_{\mathbb{P}(W)})=0$.
	
	Therefore, the sequence on the right hand column also splits. Then
	\begin{equation*}
		E|_{\mathbb{P}(W)}\cong E'\oplus \mathcal{O}_{\mathbb{P}(W)}^{(n-d-k)(m+1)}
	\end{equation*}
	with a bundle $E'$ satisfying $h^0(\mathbb{P}(W),E'^\ast)=0$.
	
	All together, we get
	\[E|_{\mathbb{P}(W)}\cong\left\{
	\begin{array}{ll}
		\mathcal{T
		}_{\mathbb{P}(W)}(-1)^{\oplus(m+1)}\oplus\mathcal{O}_{\mathbb{P}(W)}
		^{(n-d-k)(m+1)-1}& ,W_{0}\nsubseteq W ;\\
		E'\oplus \mathcal{O}_{\mathbb{P}(W)}^{(n-d-k)(m+1)}& ,W_{0}\subseteq W .\\
	\end{array}
	\right.   \]
\end{proof}

\begin{cor}
	Given the same assumptions as Lemma \ref{lemw1}, we have 	
	 \begin{equation*}
		h^0(\mathbb{P}(W),E^\ast|_{\mathbb{P}(W)})=
		\left\{
		\begin{array}{ll}
			(n-d-k)(m+1)-1&,W_{0}\nsubseteq W; \\
			(n-d-k)(m+1)&,W_{0}\subseteq W. \\
		\end{array}
		\right.
	\end{equation*}
\end{cor}
\begin{thm}\label{MThm}
For every integer $d$ in the range $1 \leq d \leq n-d-1$, there exists a vector bundle $E$ of rank $(n-d)(m+1)-1$ on $\mathbb{G}(d,n)$ with a $\mathbb{P}$-homogeneity degree of $h(E)=d$. Specifically, $E$ is uniformly of splitting type $(1,\dots,1,0,\dots,0)$ but is not nonhomogeneous.
\end{thm}
\begin{proof}
We can construct a vector bundle $E$ on $\mathbb{G}(d,n)$ using the sequence \eqref{seq1} with the property outlined in Lemma \ref{lemw1}. If $k=m-1$, it's always the case that $W_{0}\nsubseteq W$, leading to the conclusion that $E$ is $\mathbb{P}^{(m-1)}$-homogeneous and of splitting type $(1,\dots,1,0,\dots,0)$. When $m<n-d$ we can find another $m$-dimensional projective subspace $\mathbb{P}(W)$ in $\mathbb{P}^{(n-d)}$ besides $\mathbb{P}(W_0)$ such that the vector bundle $E$ restricts on them are different. Consequently, $E$ cannot be $m$-homogeneous.
Thus, for every such $(m+1)$-tuple $(w_0,w_1,\cdots,w_m)$, where $d<m<n-d$, we have successfully constructed a bundle $E$ on $\mathbb{G}(d,n)$ with $h(E)=m-1$.
In particular, if we choose $m=d+1$, we obtain a nonhomogeneous uniform bundle of rank $(d+2)(n-d)-1$ and of splitting type $(1,\dots,1,0,\dots,0)$.
\end{proof}

\section{Some nonhomogeneous uniform bundles on $\mathbb{G}(d,n)$  }
Let $p$ be an integer with $p\geq 2$, and let $H$ be a linear subspace of $S^pV$. The fiber of $S^p \mathcal R_\mathbb{G}$ at a point $x$ in $\mathbb{G}(d,n)$, denoted by $S^p \mathcal R_{\mathbb{G},x}$, is the subspace $S^pX$ of $S^pV$. We can then define a bundle morphism $f_H$ from $S^p \mathcal R_\mathbb{G}$ to $\mathcal O_{\mathbb{G}} \otimes {(S^pV/H)}$ such that the fiber morphism $f_{H,x}$ sends $S^pX$ to $S^pX+H$.
It follows that the bundle morphism $f_H$ is injective if and only if the elements in $H$ that can be expressed as elements in $S^pX$ are zero. This leads to the determination of an exact sequence:
\begin{equation}\label{new1 sequence}
	\xymatrix{
		0\ar[r] & S^p \mathcal R_\mathbb{G}\ar[r]^-{f_H}
		&\mathcal O_\mathbb{G}\otimes {(S^pV/H)} \ar[r]& E_{H}\ar[r] &0},
\end{equation}
where $E_{H}$ is the associated quotient bundle over $\mathbb{G}(d,n)$.
	
Let $l\in F(d,d+2)$ be a line in $\mathbb{G}(d,n)$. Restricting the sequence \eqref {new1 sequence} to $l$, we obtain the following exact sequence:
\begin{equation}\label{new2 sequence}
	\xymatrix{
		0\ar[r] & S^p \mathcal R_\mathbb{G}|_l\ar[r]^-{f_H|_l}
		&\mathcal O_l\otimes {(S^pV/H)} \ar[r]& E_{H}|_l\ar[r] &0}.
\end{equation}

By taking the dual of the sequence \eqref {new2 sequence}, we get an exact sequence
\begin{equation}\label{new3 sequence}
	\xymatrix{
		0\ar[r] & E_{H}^*|_l\ar[r]&\mathcal O_l\otimes {(S^pV/H)^*} \ar[r]^-{f_H^*|_l} & S^p \mathcal R_\mathbb{G}^*|_l \ar[r] &0},
\end{equation}
which induces a long exact sequence
	\begin{multline}\label{new4 sequence}
	0\rightarrow
	H^0(l,E_{H}^*|_l)\rightarrow
	H^0(l,\mathcal O_l\otimes {(S^pV/H)^*}) \xrightarrow{H^0(f_H^*|_l)}
	H^0(l,S^p \mathcal R_\mathbb{G}^*|_l) \rightarrow \\
	H^1(l,E_{H}^*|_l)\rightarrow
	H^1(l,\mathcal O_l\otimes {(S^pV/H)^*}) \xrightarrow{H^1(f_H^*|_l)}
	H^1(l,S^p \mathcal R_\mathbb{G}^*|_l)\rightarrow
	\cdots.
\end{multline}

\begin{lem}
	The bundle $E_{H}$ is uniform of splitting type $(1,\dots,1,0,\dots,0)$ if and only if, for every line $l$ in $\mathbb{G}$, $ImH^0(f_H^*|_l)^* \cap H =\{0\}$, where the morphism $H^0(f_H^*|_l)$ is induced from the cohomology of the sequence \eqref {new3 sequence}.
\end{lem}
\begin{proof}
From the sequence \eqref{new2 sequence}, we can observe that $E_{H}|l$ is generated by global sections. According to Grothendieck's theorem, $E_{H}|l$ can be expressed as $\bigoplus_{i=1}^r \mathcal O_l(a_i(l))$, where $a_i \geq 0$ for $1 \leq i \leq r$, and $r$ is the rank of $E_{H}$.

The splitting type of $E_{H}$ is $(1,\dots,1,0,\dots,0)$ if and only if $h^1(E_{H}^*|l)=0$ for every line $l$ in $\mathbb{G}$, which is equivalent to $h^0(E_{H}(-2)|_l)=0$ by Serre duality theorem.
We can identify the line $l$ with $\mathbb{P}(D)$, where the plane $D$ in $V$ is isomorphic to the space $V_{d+2}/V_d$.

%
    In the sequence (\ref{new4 sequence}), note that $H^1(l,\mathcal O_l\otimes {(S^pV/H)^*})$ equals $0$. Then, the equation $h^1(E_{H}^*|l)=0$ holds if and only if the morphism
    $$\begin{array}{*{20}{c}}\label{morphism1}
    	{H^0(f_H^*|_l):} & {H^0(l,\mathcal O_l\otimes {(S^pV/H)^*})} & {\longrightarrow} & {H^0(l,S^p \mathcal R_\mathbb{G}^*|_l)}
    \end{array}$$
    is surjective. It is evident that $H^0(l,\mathcal O_l\otimes {(S^pV/H)^*})$ is the space $(S^pV/H)^*$ and $H^0(l,S^p \mathcal R_\mathbb{G}^*|_l)$ is the space $\bigoplus\limits_{i=0}^{p} S^i D^{*\bigoplus\tbinom{d-1+p-i}{p-i}}$. The above morphism $H^0(f_H^*|_l)$
    is surjective if and only if its dual morphism
    $$\begin{array}{*{20}{c}}\label{morphism2}
    	{H^0(f_H^*|_l)^*:} & {\bigoplus\limits_{i=0}^{p}
    	S^i D^{\bigoplus\tbinom{d-1+p-i}{p-i}}} & {\longrightarrow} & {S^pV/H}
    \end{array}$$
    is injective, i.e., $ImH^0(f_H^*|_l)^* \cap H =\{0\}$. Thus we can regard $\bigoplus\limits_{i=0}^{p} S^i D^{\bigoplus\tbinom{d-1+p-i}{p-i}}$ as $ImH^0(f_H^*|_l)^*$. In particular, the dimensions of $ImH^0(f_H^*|_l)^*$ and $\bigoplus\limits_{i=0}^{p} S^i D^{\bigoplus\tbinom{d-1+p-i}{p-i}}$ are the same.
\end{proof}

For convenience, let's denote $W_l$ as the subspace $\bigoplus\limits_{i=0}^{p}S^i D^{\bigoplus\tbinom{d-1+p-i}{p-i}}$ obtained above in $S^pV$. We will use $\widetilde{V}$ for the linear space $\bigoplus\limits_{i=0}^{p}S^i V^{\bigoplus\tbinom{d-1+p-i}{p-i}}$ and $F$ for the flag $F(d,d+2)$. Let $Z$ be the union of subspaces $W_l$ in $S^pV$, i.e, $Z=\bigcup\limits_{l} W_l$, with $l$ running throughout $F$. Thus the bundle $E_{H}$ is uniform of splitting type $(1,\dots,1,0,\dots,0)$ if and only if $Z\cap H=\{0\}$.

\begin{lem}
	The set $Z$ is a closed subvariety of dimension $$(d+2)(n-d)+d-2+\sum\limits_{i=0}^p\tbinom{d-1+p-i}{p-i}(1+i)$$ in $S^PV$.
\end{lem}
\begin{proof}
	Suppose $M$ is a subbundle of the trivial bundle $F\times V$ such that
	the fiber of $M$ at a point $l$ in $F$ is $D$. Then we have a subbundle
	$S^pM$ of the trivial bundle $F\times S^pV$ with $S^pM_l=S^pD$. Moreover, $\widetilde{M}\coloneqq\bigoplus\limits_{i=0}^{p}S^i M^{\bigoplus\tbinom{d-1+p-i}{p-i}}$ is a subbundle of the trivial bundle $F\times \widetilde{V}$ with fiber $W_l$. The dimension of $\widetilde{M}$ is equal to $dim F+dimW_l $, i.e,
	$(d+2)(n-d)+d-2+\sum\limits_{i=0}^p\tbinom{d-1+p-i}{p-i}(1+i)$. Denote by $\mathbb{P}(\widetilde{M})$ and $\mathbb{P}(F\times \widetilde{V})$ the projective bundles associated to the bundles $\widetilde{M}$ and $F\times \widetilde{V}$, respectively. Considering the projection morphism
	$\begin{array}{*{20}{c}}\label{morphism5}
		{p_2:} & {\mathbb{P}(F\times \widetilde{V})} & {\longrightarrow} & {\mathbb{P}(\widetilde{V})} ,
	\end{array}$
	$\mathbb{P}(Z)$ is the image of the morphism $p_2|_{\widetilde{M}}$.
	Since $p_2$ is regular, it is also closed. Therefore $Z$ is a closed subvariety of $S^pV$ and its dimension is equal to the dimension of $S^pM$ immediately.
\end{proof}

\begin{cor}
	There is a subspace $H_p$ of codimension $$(d+2)(n-d)+d-2+\sum\limits_{i=0}^p\tbinom{d-1+p-i}{p-i}(1+i)$$ in $S^pV$ such that $Z\cap H_p=\{0\}$. The associated quotient bundle $E_{H_p}$ constructed as \eqref {new1 sequence}  is uniform of splitting type
	$(1,\dots,1,0,\dots,0)$ and of rank $$(d+2)(n-d)+d-2+\sum\limits_{i=0}^p\tbinom{d-1+p-i}{p-i}(1+i)-\tbinom{p+d}{p}.$$
\end{cor}

Next, we will show that $E_{H_p}$ is nonhomogeneous.

\begin{thm}[Chow]\label{thmz1}
	$Aut(\mathbb{G}(d,n))=PGL(V)$ for $2d\neq n-1$. $PGL(V)$ is a normal subgroup of index 2 in $Aut(\mathbb{G}(d,n))$ for $2d=n-1$.
\end{thm}
\begin{proof}
	See\cite{cowen1989automorphisms},Theorem 1.1.
\end{proof}

\begin{lem}\label{lemz3}
	For any $t\in Aut(\mathbb{G}(d,n))$ which is isomorphic to an element of $PGL(V)$, i.e., $t=T/\sim$ with $T\in GL(V)$ and scalar multiplicative relation $\sim$, $t^*E_{H_p}$ is isomorphic to $E_{(S^pT)^{-1}(H_p)}$.
\end{lem}
\begin{proof}
	The associated quotient bundle $E_{H_p}$ constructed as \eqref {new1 sequence} is determined by the sequence:
	\begin{equation*}\label{new6 sequence}
		\xymatrix{
			0\ar[r] & S^p \mathcal R_\mathbb{G}\ar[r]^-{f_{H_p}}
			&\mathcal O_\mathbb{G}\otimes {(S^pV/H_p)} \ar[r]& E_{H_p}\ar[r] &0}.
	\end{equation*}
	It follows that $t^*E_{H_p}$ is the cokernel of the morphism
	$$\begin{array}{*{20}{c}}\label{morphism8}
		{t^*f_{H_p}:} & {t^*S^p \mathcal R_\mathbb{G}} & {\longrightarrow} & {\mathcal O_\mathbb{G}\otimes {(S^pV/H_p)}}.
	\end{array}$$
	Since $S^p \mathcal R_\mathbb{G}$ is homogeneous and the linear subspace $H_p$ becomes the subspace $(S^pT)^{-1}(H_p)$ induced by the morphism $(S^pT)^{-1}\in GL(S^pV)$, we have the following commutative diagram
	\begin{equation*}
	\begin{tikzcd}[column sep=huge]
		t^*S^p \mathcal R_\mathbb{G}\arrow[r,"t^*f_{H_p}"]\arrow[d,"\phi"]& \mathcal O_\mathbb{G}\otimes {(S^pV/H_p)}\arrow[d,"\psi"] \\
		S^p \mathcal R_\mathbb{G}\arrow[r,"f_{(S^pT)^{-1}(H_p)}"] &\mathcal O_\mathbb{G}\otimes {(S^pV/(S^pT)^{-1}(H_p))}		
	\end{tikzcd}
	\end{equation*}
	with $\phi$ and $\psi$ are both isomorphic. The associated quotient bundle $E_{(S^pT)^{-1}(H_p)}$ can be constructed as the morphism $f_{(S^pT)^{-1}(H_p)}$ is injective, i.e., $E_{(S^pT)^{-1}(H_p)}$ is the cokernel of $f_{(S^pT)^{-1}(H_p)}$. Furthermore, $t^*E_{H_p}\cong E_{(S^pT)^{-1}(H_p)}$.
\end{proof}

\begin{lem}\label{lemz4}
	Suppose $H$ and $H^\prime$ are two linear subspaces of $S^pV$ such that the morphisms $f_H$ and $f_{H^\prime}$ are both injective. The associated quotient bundles $E_H$ and $E_{H^\prime}$ are isomorphic if and only if $H$ is equal to $H^\prime$.
\end{lem}

\begin{proof}
	It is clear that $E_H$ and $E_{H^\prime}$ are isomorphic if $H = H^\prime$. Conversely, let
	$$\begin{array}{*{20}{c}}\label{morphism6}
		{g:} & {E_H} & {\longrightarrow} & {E_{H^\prime}}
	\end{array}$$
	be a bundle isomorphism. Taking the cohomology of the sequence \eqref {new1 sequence}, we have the following exact sequence
	\begin{equation*}\label{new5 sequence}
		0\to  H^0(\mathbb{G},S^p \mathcal R_\mathbb{G})\to
		H^0(\mathbb{G},\mathcal O_\mathbb{G}\otimes {(S^pV/H)}) \to H^0(\mathbb{G},E_{H})\to  H^1(\mathbb{G},S^p \mathcal R_\mathbb{G})\to  0.
	\end{equation*}
	Since $H^1(\mathbb{G},S^p \mathcal R_\mathbb{G})=0$ under the assumption conditions $1\le d\le n-d-1$ , we can identify
	$H^0(\mathbb{G},E_{H})$ with $S^pV/H$. Then we get an isomorphism
	$$\begin{array}{*{20}{c}}\label{morphism7}
		{H^0(g):} & {S^pV/H} & {\longrightarrow} & {S^pV/H^\prime}
	\end{array}$$
	and the following commutative diagram
	\begin{displaymath}
		\xymatrix{
			0\ar[r]& S^p\mathcal R_\mathbb{G}\ar[r] \ar[d]^{\cong} & \mathcal O_\mathbb{G}\otimes {(S^pV/H)}\ar[r] \ar[d]^{I_{\mathcal O_\mathbb{G}}\bigotimes(H^0(g))} & E_H\ar[r] \ar[d]^{g} &0 \\
			0\ar[r]& S^p\mathcal R_\mathbb{G}\ar[r] & \mathcal O_\mathbb{G}\otimes {(S^pV/H^\prime)}\ar[r]  & E_{H^\prime}\ar[r] & 0\\
		}
	\end{displaymath}
	with three isomorphic vertical arrows. Considering the middle terms of the above diagram, we obtain $H = H^\prime$ by restricting to $H$ or $ H^\prime$ on their fibers.
\end{proof}

\begin{lem}\label{lemz5}
	For any proper linear subspace $H$ of $S^pV$, there exists a linear transformation $T$ of $V$ such that $H$ and $(S^pT)H$ are distinct .
\end{lem}
\begin{proof}
	See \cite{drezet1980exemples} Lemma 4.
\end{proof}

\begin{prop}
	The associated quotient bundle $E_{H_p}$ on $\mathbb{G}(d,n)$ is nonhomogeneous.
\end{prop}
\begin{proof}
	According to the Theorem \ref {thmz1}, we have known that every element $t$ in $ Aut(\mathbb{G}(d,n))$ can be represented by a linear transformation $T$ in $PGL(V)$ when $2d\neq n-1$. Then $t^*E_{H_p}$ is not isomorphic to $E_{H_p}$ as the result of Lemma \ref {lemz3}, Lemma \ref {lemz4} and Lemma \ref {lemz5}. For the case when $2d=n-1$,  $Aut(\mathbb{G}(d,n))$ contains all the elements in $PGL(V)$. The result also holds for the same reason as the case when $2d\neq n-1$.
\end{proof}

\begin{prop}
	The associated quotient bundle $E_{H_p}$ on $\mathbb{G}(d,n)$ is simple thus indecomposable.
\end{prop}
\begin{proof}
	We only need to prove $h^0(\mathbb{G},E^*_{H_p}\otimes E_{H_p})=1$.
	Taking the dual of the exact sequence \eqref {new1 sequence}, we have an exact sequence:
	\begin{equation}\label{new7 sequence}
		\xymatrix{
			0\ar[r] & E_{H_p}^* \ar[r]
			&\mathcal O_\mathbb{G}\otimes {(S^pV^*/H_p^*)} \ar[r]&  S^p \mathcal R^*_\mathbb{G} \ar[r] &0}.
	\end{equation}
	Taking the cohomology of the above exact sequence, we can compute
	$$H^i(\mathbb{G},  E_{H_p}^*)= \left\{
	\begin{array}{ll}
		H_p^* &  ,\ i=1 \\
		0 &  ,\ otherwise \\
	\end{array} .
	\right.   $$
	Applying  the functor $-\otimes S^p \mathcal R_\mathbb{G}$ to the exact sequence \eqref{new7 sequence},
	we obtain the following exact sequence
    \begin{equation*}
	\xymatrix{
		0\ar[r] & E_{H_p}^*\otimes S^p \mathcal R_\mathbb{G} \ar[r]
		& S^p \mathcal R_\mathbb{G}\otimes {(S^pV^*/H_p^*)} \ar[r]&  S^p \mathcal R^*_\mathbb{G}\otimes S^p \mathcal R_\mathbb{G} \ar[r] &0}.
    \end{equation*}
    Then $h^1(\mathbb{G},E^*_{H_p}\otimes S^p \mathcal R_\mathbb{G})=h^1(\mathbb{G},S^p \mathcal R^*_\mathbb{G}\otimes S^p \mathcal R_\mathbb{G})=1$, as the result of the proposition 2.2 in \cite{Kapranov1985ONTD}.
    Furthermore the short exact sequence
    \begin{equation*}
    	\xymatrix{
    		0\ar[r] & E_{H_p}^*\otimes S^p \mathcal R_\mathbb{G} \ar[r]
    		& E_{H_p}^*\otimes {(S^pV/H_p)} \ar[r]&  E_{H_p}^*\otimes E_{H_p} \ar[r] &0}
    \end{equation*}
	yields the following long exact sequence:
	\[\cdots \longrightarrow H^0(\mathbb{G},E_{H_p}^*\otimes {(S^pV/H_p)})
	\longrightarrow H^0(\mathbb{G},E_{H_p}^*\otimes E_{H_p})
	\longrightarrow H^1(\mathbb{G},E_{H_p}^*\otimes S^p \mathcal R_\mathbb{G})\longrightarrow \cdots .\]
	It follows that $h^0(\mathbb{G},E^*_{H_p}\otimes E_{H_p})=1$.
\end{proof}

 To sum up, we have the following results.
\begin{thm}
There exists an indecomposable uniform vector bundle of rank $$(d+2)(n-d)+d-2+\sum\limits_{i=0}^p\tbinom{d-1+p-i}{p-i}(1+i)-\tbinom{p+d}{p}$$ with splitting type
$(1,\dots,1,0,\dots,0)$ but nonhomogeneous over Grassmannian $\mathbb{G}(d,n)$.
	\end{thm}

Finally, we propose the following problem for the interested readers.
\begin{prob}
Determine the uniform-homogeneous shreshold $UH(X)$ for rational homogeneous spaces $X$ of Picard number $1$.
\end{prob}

\bibliography{ref}
\nocite{*}
\bibliographystyle{plain}

\end{document}